\documentclass[11pt]
{article}
\usepackage{amsmath,amsfonts,amssymb,amsthm,amscd}
\title{Products of straight spaces}
\date{}
    \author{Alessandro Berarducci \and Dikran Dikranjan\thanks{This author
acknowledges the financial aid received from MCYT, MTM2006-02036 and
FEDER funds.
\endgraf AMS classification numbers:  Primary 54D05;
    Secondary 54C30. \endgraf
   }   \and Jan Pelant}

\newtheorem{theorem}{Theorem}[section]
\newtheorem{corollary}[theorem]{Corollary}

\newtheorem{lemma}[theorem]{Lemma}
\newtheorem{proposition}[theorem]{Proposition}
\newtheorem{problem}[theorem]{Problem}

\newtheorem{em-example}[theorem]{Example}
\newtheorem{em-def}[theorem]{Definition}
\newtheorem{em-remark}[theorem]{Remark}
\newtheorem{em-question}[theorem]{Question}
\newenvironment{example}{\begin{em-example} \em }{ \end{em-example}}
\newenvironment{definition}{\begin{em-def} \em  }{ \end{em-def}}
\newenvironment{remark}{\begin{em-remark} \em }{\end{em-remark}}
\newenvironment{question}{\begin{em-question}
\em }{\end{em-question}}
\newenvironment{pf}{{\em Proof of claim.}}{{\em End of Claim. \medskip}}
\newtheorem{Claim}[theorem]{Claim}

\newcommand{\R}{{\mathbb R}}
\newcommand{\N}{{\mathbb N}}
\newcommand{\Q}{{\mathbb Q}}

\newcommand{\ov}{\overline}

\newcommand{\nin}{\not\in}
\newcommand{\eps}{\varepsilon}
\newcommand{\nbd}{neighbourhood }

\newcommand{\diam}{\mbox{diam } }

\newcommand{\ttt}{{tight}}

\newcommand{\NB}{}
\newcommand{\AB}{}

\makeindex

\begin{document}
\maketitle

\begin{abstract} A metric space $X$ is straight if for each finite
cover of $X$ by closed sets, and for each real valued function $f$ on $X$, if
$f$ is uniformly continuous on each set of the cover, then $f$ is uniformly
continuous on the whole of $X$. A locally connected space is straight iff
it is uniformly locally connected (ULC). It is easily seen that 
 ULC spaces are stable under finite products. On the
other hand the product of two straight spaces is not necessarily straight.  We
prove that the product $X\times Y$ of two metric spaces is straight if and only
if both $X$ and $Y$ are straight and one of the following conditions holds:
\begin{itemize}
\item[(a)] both $X$ and $Y$ are precompact;
\item[(b)] both $X$ and $Y$ are locally connected;
\item[(c)] one of the spaces is both precompact and locally connected.
\end{itemize}
   \par In particular, when $X$ satisfies (c), the product $X\times Z$ is straight {\em for every} straight space $Z$. 
\par Finally, we characterize when  infinite products of metric  spaces are ULC and  we completely solve the problem of straightness of infinite products of ULC spaces.
\end{abstract}

\section{Introduction}
All spaces in the sequel are metric. Given a space $X$, $C(X)$ denotes the set
of all continuous functions $f\colon X \to \R$. The following notion, already
studied in \cite{BDP3,BDP+}, will be the object of investigation of this paper.

\begin{definition}
A space $X$ is called {\bf straight} if whenever $X$ is the union of finitely
many closed sets, then $f \in C(X)$ is uniformly continuous (briefly, u.c.) if
and only if its restriction to each of the closed sets is u.c.
\end{definition}

\begin{example} Every compact space is obviously straight. For the same reason every UC-space is straight (a space $X$ is UC if every $f\in
    C(X)$ is uniformly continuous \cite{A}).
\end{example}

More examples are obtained from the following stronger property

\begin{definition}\label{SCS_uniformly locally connected} ([HY, 3-2]) A metric space $X$ is {\bf uniformly locally connected} (ULC), if for
    every $\varepsilon > 0$ there is $\delta > 0$ such that any two points at distance $<\delta$ lie in a connected set of diameter
    $<\varepsilon$.
\end{definition}

It is easy to see that a ULC space is locally connected, namely each point has a basis of connected neighbourhoods. Therefore a compact space need not be ULC. 

\begin{theorem}\label{Th_loc_con_vs_ULC}
(\cite[Theorem 3.9]{BDP3}) A locally connected metric space is straight if and only if it is uniformly locally connected.
\end{theorem}

 In particular $\R$ is straight and every topological vector space is straight.  The circle minus a point is not straight (it is locally connected but not uniformly so). As far as non-locally connected spaces are concerned, $\Q$ is not straight. More generally a totally disconnected space is straight if and only if it is a UC-space (\cite[Theorem 4.6]{BDP3}).

One of the main results of \cite{BDP+} is a characterization of the complete straight spaces in terms of the properties of the 
quasi-components of the space and its subspaces (see Corollary \ref{coro_comp_try} and Definition \ref{deadend} below).

If a product $X\times Y$ is straight then both $X$ and $Y$ are straight, but the converse is not true in general (i.e., the class of straight spaces is not closed under finite products).
   One of the main goals of the paper is to establish precisely  when straightness is preserved under products.

As a first step we show that the class of ULC spaces behaves better in this respect: it is included in the class of straight
spaces and it is stable under finite products (Lemma~\ref{prod_ULC}). Moreover, if $X$ is a precompact ULC space, then $X\times Y$ is straight for every
straight space $Y$ (and this property characterizes the precompact ULC space,
cf. Theorem \ref{COR.pre-ULC}), i.e., the precompact ULC spaces have the best possible behavior with respect to productivity.

The failure of the corresponding property for straight spaces can be witnessed as follows: if $K$ is a totally disconnected 
compact space (e.g. the Cantor space), then $\R \times K$ is not straight although both factors are straight. This follows 
from the following curious {\em dichotomy}: if a product $X\times Y$ is straight, then either $X$ is precompact or
$Y$ is uniformly locally connected (Corollary \ref{Thm_loc_conn}). This implies the ``only if'' direction in the following theorem that completely describes when straightness is available for a product of two spaces.

\bigskip

\noindent {\bf Theorem A.} {\em The product $X\times Y$ of two metric spaces
is straight if and only if both $X$ and $Y$ are straight and one of the
following conditions holds:
\begin{itemize}
\item[(a)] both $X$ and $Y$ are precompact;
\item[(b)] both $X$ and $Y$ are ULC;
\item[(c)] one of the spaces is both precompact and ULC.
\end{itemize}}

   \bigskip
The sufficiency of (b) and (c) was already commented above. To the proof of the sufficiency of (a) is dedicated the entire \S \ref{suff(a)}. 
The proofs use essentially criterions (Theorem \ref{str-dense} and Lemma \ref{dense-super}) for straightness of dense subspaces based on the notion of a {\em tight extension}
(this is specific form of dense embedding introduced in \cite{BDP+}, see Definition \ref{def tight}). The class of tight embeddings has many nice properties that could be useful in other situations (see Theorem \ref{NewTh} and \ref{pulc}, as well as the comment in the last section). 
In Theorem \ref{NewTh} we establish first a natural general property of the class of tight maps: they are closed under finite products. As a corollary we obtain the sufficiency of (a) (Theorem~\ref{yama-ales}). To resume, the proof of Theorem A is contained in Corollary \ref{Thm_loc_conn}, Lemma \ref{prod_ULC}, 
Theorem \ref{yama-ales} and Theorem \ref{COR.pre-ULC}. This theorem was announced without proof in \cite{BDP+}. 

For reader's convenience we formulate explicitly the following immediate corollary  from Theorem A: 

\bigskip

\noindent {\bf Corollary 1.} {\em Let $X_1, \ldots, X_n$ be metric spaces. Then $X = \prod_{i=1}^n X_i$
is straight if and only if all spaces $X_i$ are straight and one of the
following conditions holds:
\begin{itemize}
   \item[(a)] all spaces $X_i$ are precompact;
   \item[(b)] all spaces $X_i$ are ULC;
   \item[(c)] all but one of the spaces are both precompact and ULC.
\end{itemize}}

\bigskip

The following fact on straightness of products of two spaces is established also by Nishijima and  Yamada \cite{Y}  if $X \times (\omega + 1)$ is straight for some metric space $X$, then $X$ is precompact and $X \times K$ is straight for every compact space $K$ (see Example \ref{exYamada} and Corollary \ref{yam} for more details).

Finally, in \S 6 we face the problem of straightness of infinite products of spaces and
we completely solve the problem of straightness of infinite products of ULC spaces: 

\bigskip

\noindent {\bf Theorem B.} {\em  Let $X_n$ be a ULC space for each $n\in \N$ and $X = \Pi_n X_n$. 
\begin{itemize}
  \item[(a)] $X$ is ULC iff all but finitely many $X_n$ are connected.
  \item[(b)] The following are equivalent:
\begin{itemize}
    \item[(b$_1$)] $X$ is straight.
    \item[(b$_2$)] either $X$ is ULC or each $X_n$ is precompact.
\end{itemize}
\end{itemize}}

\bigskip

This theorem is proved at the end of \S 6. In particular, the theorem completely settles the case of infinite powers of ULC space:

\bigskip

\noindent {\bf Corollary 2.} {\em Let $X$ be ULC. Then 
\begin{itemize}
   \item[(a)] $X^\omega $ is ULC iff $X$ is connected; 
   \item[(b)] $X^\omega$ straight iff $X$ is either connected or precompact.
\end{itemize}}

As far as straightness of infinite products $X = \Pi_n X_n$ is concerned, we prove in Proposition \ref{NecessInf} that straightness of 
$X$ implies the straightness of each space as well as the disjunction of the condition (b$_2$) (from Theorem B) and the following one:

(i) all but one of the spaces are both precompact and ULC and all but finitely many of the spaces are connected.

While (i) is easily seen to be also sufficient (see Remark \ref{NewRemark}), we do not know whether a product of infinitely many precompact straight spaces
is straight (see Question \ref{Ques-Inf}). 

\bigskip

\noindent{\sc Acknowledgements:} It is a pleasure to thank the referee for her/his very careful reading
and numerous useful suggestions. 

\section{Background}\label{PP}

\bigskip
\noindent {\bf Notations.}
\begin{enumerate}
\item We identify $\omega+1$ with a compact subset of $\R$ of order type $\omega + 1$, namely with an increasing converging sequence together with its limit point.
\item Usually a metric space $X$ with metric $d$ will be denoted by $(X,d)$. In the presence of more spaces $X$, $Y$, we will use subscprits $d_X$, $d_Y$ to avoid confusion.   
\item As we are interested in the uniform
properties of metric spaces, we can assume that metrics are bounded by $1$ to avoid unnecessary difficulties.
  \item Unless otherwise stated, the metric $d(x,y)$ on a product $\prod_{i=1}^n X_i$ of finitely many metric spaces $(X_i,d_i)$ is
defined as the sum $\Sigma_i d_i(x_i,y_i)$, where $x_i,y_i$ are the coordinates of $x,y$. In the case of an infinite 
(countable) product $\prod_{i=1}^\infty X_i$, one has to 
start with uniformly bounded metrics $d_n$ (see the remark in the 
previous item) and define $d(x,y) = \Sigma_n \frac 1{2^n} d_n(x_n,y_n)$ where $x$ and
$y$ are points from the product $ \prod_{i=1}^n X_n$ and $x_i$ and $y_i$ are corresponding coordinates. 
  \item We will frequently subscripts like $C^+_\eps$, $C^-_\eps$ and variants of it (e.g. $A_\eps, B_\eps$ where $A,B$ is a given binary cover of a space). Such notation refers to Definition \ref{notation}. 
  \item The ball of center $x$ and radius $\eps$ in a metric space $(X,d)$ is denoted by $B_\eps(x)$. If the metric is not
clear from the context we also use the notation $B^d_\eps(x)$. For a metric space $M$, we use also $B^M_\eps (x)$; it can be convenient
when we deal with a space and its subspaces.
\end{enumerate}

We recall here some non-trivial facts from \cite{BDP3} which will be often used in the sequel.

In the definition of ``straight'' it suffices to consider only binary unions:

\begin{theorem}(\cite{BDP3}) A space $X$ is straight if and only if whenever $X$ is the union of
two closed sets, then $f \in C(X)$ is u.c. if and only if its restriction to each of the closed sets is u.c.
\end{theorem}

Using this characterization one can prove the following necessary and sufficient condition for straightness. We need first a definition.

\begin{definition}\label{notation}
Let $(X,d)$ be a metric space. A pair $C^+,C^-$ of closed sets of $X$ is {\em u-placed} if $d(C^+_\eps, C^-_\eps)>0$ holds for every $\eps>0$, where
$C^+_\eps=\{x\in C^+: d(x,C^+\cap C^-)\geq \eps\}$ and $C^-_\eps=\{x\in C^-: d(x,C^+\cap C^-)\geq \eps\}$.
\end{definition}

In other words $C^+,C^-$ is u-placed if for every pair of sequences $x_n \in C^+$ and $y_n \in C^-$ with $d(x_n,y_n) \to 0$, we have
$d(x_n, C^+\cap C^-) \to 0$ (for $n \to \infty$). In particular, if $C^+\cap C^-=\emptyset$, then $C^+,C^-$ is u-placed iff $d(C^+, C^-)>0$.

\begin{theorem}\label{uplaced} (\cite[Corollary 2.10]{BDP3}) A metric space $(X,d)$ is straight if and only if
every pair of closed subsets, which form a cover of $X$, is u-placed.
\end{theorem}

\begin{corollary}\label{clopen-str} (\cite{BDP+}) If a metric space $(X,d)$ is straight and a proper subset $H \subset X$ is clopen, then the distance between $H$ and $X\setminus H$ is positive.
\end{corollary}

Now we will need the following equivalent description of ULC spaces: 

\begin{lemma}\label{unif.l.c.} A metric space $(X,d)$ is ULC if and only if for each $\eps >0$ there is a positive $\delta$ such that for each $x\in X$, there is an open connected set $W_x$ such that
   \begin{equation}\label{obvious}
      B_\delta (x)\subseteq W_x\subseteq B_\eps (x).
\end{equation}
\end{lemma}

Without the requirement of openness of $W_x$ this is \cite[Lemma 3.1]{BDP+}. It remains to observe that once a connected set   $W_x$ with the above property is found, one can use local connectedness of the space to replace $W_x$ by a larger  set $W_x^*$ still contained in $B_\eps (x)$ that is both connected and open. 

Another group of results from \cite{BDP+} we are going to use here concerns preservation of straightness under extensions. The following property of extensions will be crucial.

\begin{definition}\label{def tight}(\cite{BDP+}) An extension $X\subseteq Y$ of topological spaces is called {\bf tight} if for
every closed binary cover $X=F^+\cup F^-$ one has
\begin{equation}\label{def_ttt}
     \overline{F^+}^Y\cap \overline{F^-}^Y=\overline{F^+\cap F^-}^Y.
\end{equation}
\end{definition}

Let us note that even a one-point extension can easily fail to be tight: take $X=\{1/n:n\in \N\}$, $Y=X\cup \{0\}$ and as $F^+, F^-$ the subsequences with
even and odd indices respectively. Examples of tight extensions are provided by the following

\begin{theorem}\label{str-dense}(\cite{BDP+}) Let $X,\ Y$ be metric spaces, $X\subseteq Y$ and let $X$ be dense in $Y$. Then
$X$ is straight if and only if $Y$ is straight and the extension $X\subseteq Y$
is \ttt. \end{theorem}

Since the tightness of an extension $X\subseteq Y$ is equivalent to the joint tightness of all one-point extensions $X\subseteq X \cup \{y\}$, 
$y\in Y \setminus X$, the theorem implies that an extesion $Y$ of $X$ is straight iff the one-point extensions $X \cup \{y\}$ 
are  straight for all $y\in Y \setminus X$. 

By the theorem (and the corollary below) every non-complete straight space has a proper tight extension. 

\begin{corollary}\label{NEWcorol}  Let $X$ be a metric space. Then $X$ is straight if and only if its completion  $\widetilde X$ is straight and $\widetilde X$ is a tight extension of $X$.
\end{corollary}

Let us recall some facts from \cite{BDP+} for easier reference:

\begin{lemma}\label{dense-super} \label{ulc-dense}\label{completionULC} \label{simple-claim2}
Let $X\subseteq Y$ be dense in $Y$.
  \begin{enumerate} 
  \item If $X$ is ULC, then $Y$ is ULC as well (and $Y$ is a tight extension of $X$). In particular the completion of a ULC space is ULC.
  \item If $Y$ is ULC, then the following are equivalent:
 \begin{description}
      \item[(i)] $X$ is ULC
      \item[(ii)] $X$ is straight
      \item[(iii)] $Y$ is a tight extension of $X$
 \end{description}
\end{enumerate}
\end{lemma}

 The next construction shows that the property ULC can be easily lost under passage to closed subspaces. 

\begin{example}\label{EmbedULC}  Let $(X,d)$ be a metric space. Then $X$ is homeomorphic to a closed subspace of a ULC space.
\end{example}

\begin{proof} We will construct a space $X'\supset X$ such that each pair of points $x,y\in X$ lie in a connected set $I_{x,y}\subset X'$ of
diameter $d(x,y)$. This is easy to do as follows. Fix a linear ordering of $X$ and for each pair of points $x<y$ in $X$ consider
a space $I_{x,y}$ isometric to a closed interval of $\R$ of length $d(x,y)$. Let $X'$ be the topological space $(X \cup
\bigcup_{x<y}I_{x,y})/E$ where $E$ is the equivalence relation on the disjoint union $X \cup \bigcup_{x<y}I_{x,y}$ which identifies
one of the extremes of $I_{x,y}$ with $x$ and the other with $y$. In this way $X$ is naturally identified with a subspace of $X'$
via $x\mapsto [x]$, where $[x]$ is the class of $x$ modulo $E$. Moreover $X$ is a closed subspace of $X'$ because $X'\setminus X$
is homeomorphic to a disjoint union of open intervals which are open subsets of $X'$. The metric on $X'$ is the biggest possible
compatible with the fact that the inclusion $X\subset X'$ is an isometry and that $I_{x,y}$ is isometric to an interval of $\R$ of
length $d(x,y)$. To finish the proof we show that $X'$ is ULC. Let $u,v \in X'$. Then for some $x,y, x',y'\in X$ we have $u\in
I_{x,y}$ and $v \in I_{x',y'}$ (with the natural identification of $I_{x,y}, I_{x',y'}$ as subsets of $X'$). The set $W = I_{x,y}
\cup I_{x',y'} \cup I_{x,x'} \cup I_{y,y'}$ is connected, contains $u$ and $v$, and has diameter $\leq d(x,y) + d(x',y') + d(x,x') +
d(y,y')$. A case analysis shows that $W$ contains a connected subset, still containing $u,v$, and of diameter $d(u,v)$. This
proves that $X'$ is ULC. \NB \end{proof}
  
\begin{definition} A sequence $(x_n)_{n\in \N}$ in a metric space $(X,d)$ is {\bf discrete} if it has no accumulation points 
in $X$, and it is {\bf uniformly discrete} if there is a non-zero lower bound to the set of all the distances $d(x_n,x_m)$ for 
$n\neq m$. Two sequences $(a_n)_{n\in \N}$ and $(b_n)_{n\in \N}$ are {\bf adjacent} if $d(a_n,b_n) \to 0$ for $n\to \infty$.
\end{definition}

\begin{example}\label{Ex_UC} \NB In terms of the above definition a space $X$ is UC iff for every pair $(a_n)_{n\in \N}$ and $(b_n)_{n\in \N}$ of adjacent
sequences in $X$, with $a_n\ne b_n$ for all $n\in \N$, there exists a (common) accumulation point in $X$. 
\end{example}

 According to \cite{BDP+}, a space $(X,d)$ is {\bf weakly uniformly locally connected} (WULC) if for each pair of {\it discrete} adjacent sequences $(a_n)_{n\in\N}$ and $(b_n)_{n\in\N}$ in $X$, there exists a sequence $(C_n)_{n\in\N}$ of connected subsets of $X$ and $k\in \N$ such that $\lim _{n\to \infty} \mbox{diam} C_n=0$, and $a_{n+k}\in C_n, b_{n+k}\in C_n$ for every $n\in\N$. It follows from from the definitions that ULC implies WULC.

Now we recall another notion of connectedness introduced in \cite{BDP+} weaker than WULC but strong enough to imply straightness. 
To this end we recall that the {\bf quasi-component} of a point $x\in X$ is the intersection of all clopen sets containing $x$.
Hence $x$ is in the same quasi-component of $y$ in $X$ if $x$ cannot be separated from $y$, i.e. for every partition $X = A \cup B$ with $A, B$ open, $x$ and $y$ lie both in $A$ or both in $B$. One can define a metric $\hat{d}$ as follows:

\begin{definition}\cite{BDP+}  Given a metric space $(X, d)$ and $x,y \in X$ we say that $I\subset X$ {\bf quasi-connects} $x$ and $y$ if $x,y$ belong to $I$ and are in the
 same quasi-component of $I$. We define $\hat{d}(x,y)$ as the  minimum between $1$ and the infimum of the diameters of the 
 subsets $I$ of $X$ which quasi-connect $x$ and $y$. So $\hat{d}(x,y) =1$, if there is  no set $I$ quasi-connecting $x$ and $y$.
\end{definition}

 The next definition introduces a notion of connectedness between WULC and straightness.

\begin{definition}\label{deadend}\cite{BDP+}  A metric space $(X,d)$ is {\bf approximatively locally connected}  (ALC) if for each pair of {\it discrete} adjacent sequences  $(a_n)_{n\in\N}$ and $(b_n)_{n\in\N}$, $\lim _{n\to \infty}  \hat d (a_n ,b_n) =0$. \end{definition}

Clearly, every compact space is ALC since it does not contain discrete sequences.


\begin{theorem} \label{alc-stra} (\cite{BDP+}) $UC$ $\Rightarrow$ WULC $\Rightarrow$ ALC $\Rightarrow$ straight.\end{theorem}

The next statement shows the importance of the ALC property:

\begin{corollary}\label{coro_comp_try} A complete space $X$ is straight if and only if it is ALC. \end{corollary}



Even if we are not going to use it in the sequel, let us note that if a dense  subspace $X$ of a space $ Y$ is ALC, then $Y$ itself is ALC \cite{BDP+}.
In particular, the completion of an ALC space is ALC.


\section{First properties related to products}

We show in this section and the next one the important role played by products in questions related to straightness. To mention at least one group of results, the behavior of ULC,  ALC and straightness with respect to finite powers is treated in Lemma~\ref{prod_ULC}, Proposition~\ref{cor-prec-power} and Corollary~\ref{powers} respectively.

\begin{definition} A u.c. map $f:X\to Y$ of metric spaces is said to  {\bf allow lifting of adjacent sequences} if for every pair of adjacent sequences $(x_n)$ and $(y_n)$ in $Y$, there exist 
subsequences $(x_{n_k})$ and $(y_{n_k})$ and two adjacent sequences $x_k'$ and $y_k'$ in $X$ such that $f(x_k')=x_{n_k}$ and $f(y_k')=y_{n_k}$ for every $k$.
\end{definition}

One can easily prove:

\begin{lemma}\label{lifting} Let $f:X\to Y$ be a map of metric spaces that allows lifting of adjacent sequences. Then a function $g:Y\to \R$ is u.c. iff the function
$g\circ f$ is u.c.
\end{lemma}

\begin{example}\label{Example1} There are two relevant instances of maps allowing lifting of adjacent sequences: (i) projections in products; (ii) continuous
    open homomorphisms between metric topological groups \cite{DP_STG}.
\end{example}

\begin{lemma}\label{lemma_imag}
If $f:X\to Y$ is a map of metric spaces that allows lifting of adjacent
sequences and $X$ is straight, then also $Y$ is straight.
\end{lemma}

\begin{proof} Assume $Y=F^+\cup F^-$ is a closed binary cover of $Y$. Then $X=f^{-1}(F^+)\cup f^{-1}(F^-)$ is a closed binary cover of
    $X$. Now if $g:Y\to \R$ is a continuous function such that $g|_{F^+}$ and $g|_{F^-}$ are u.c., then $f_1=g\circ f:X\to \R$ is continuous and $f_1|_{f^{-1}(F^+)}$ and $f_1|_{f^{-1}(F^-)}$ are u.c. as compositions of u.c. functions. Then $f_1$ is u.c. since $X$ is straight. Now $g$ is u.c. by Lemma \ref{lifting}.
\end{proof}

The next corollaries follow from Lemma \ref{lemma_imag} and Example
\ref{Example1}.

\begin{corollary}\label{Coro1} Let $X, Y$ be metric spaces. If $X\times Y$ is straight, then both $X$ and $Y$ are straight.
\end{corollary}

A subspace $Y$ of a metric space $X$ is said to be a {\bf uniform retract} of $X$ if there exists a u.c. map $r:X\to Y$ such that $r\restriction_Y=id_Y$.

\begin{corollary}\label{u.retract} Uniform retracts of a straight space are straight. \end{corollary}

\begin{proof}
    Let $r:X\to Y$ be a uniform retraction. Then $r$ allows lifting of adjacent sequences so that Lemma \ref{lemma_imag} applies.
\end{proof}

Here {\em uniform retract} cannot be replaced by the weaker property {\em $C_u$-embedded subspace} (i.e., a subspace $Y$ of $X$
such that every u.c. $f:Y\to \R$ can be extended to a u.c. function $X\to \R$). Take $X=\R_+\times \R$ and $Y$ = the two branches of
the hyperbola $\pm xy=1$ in $X$. 

The next corollary follows directly from Corollary \ref{u.retract} since uniformly clopen subspaces are uniform retracts. Moreover, each clopen proper
subset of a straight space must have a positive distance of its complement (see
Corollary~\ref{clopen-str}), hence each clopen subset of a straight space is
uniformly clopen.

\begin{corollary}\label{clopen} Clopen subspaces of straight spaces are straight.
\end{corollary}

The ULC spaces form a class of straight spaces stable under products:

\begin{lemma}\label{prod_ULC}
A product $X\times Y$ is ULC if and only if both $X$ and $Y$ are ULC.
\end{lemma}

\begin{proof}    If the product $X\times Y$ is ULC, then by Corollary \ref{u.retract} and Theorem \ref{Th_loc_con_vs_ULC} $X$ and $Y$ are ULC (as local connectedness
is preserved under the projections of the product). On the other
    hand, assume that $a_n=(x_n,x_n')\in X\times Y$ and $b_n=(y_n,y_n')\in X\times Y$ are adjacent sequences in $X\times Y$,
    i.e., $d(a_n,b_n)\to 0$. Find connected sets $C_n$ and $B_n$ in $X$
    and $Y$ respectively, containing $\{x_n, x_n'\}$ and $\{y_n, y_n'\}$ respectively, with $\diam (B_n)\to 0$ and $\diam (C_n) \to 0$. Then
    $C_n\times B_n$ is connected and $\diam (C_n \times B_n) \to 0$, witnessing that $X\times Y$ is ULC.
\end{proof}

We conclude the section proving a fact (Lemma~\ref{claim2} below) which provides a wide supply of tight extensions. We shall prove a
more general result in Theorem ~\ref{NewTh}, nevertheless, we prefer to give a direct (shorter) proof in this particular case.

\if
Lemma~\ref{claim2} says that tight extensions are even preserved by products
when the assumptions of Lemma~\ref{simple-claim2} are satisfied.
This should be compared with Lemma~\ref{secret}.
\fi

\begin{lemma} \label{claim2} Let $Y$ be a metric space and let $X\subseteq Y$ be dense in $Y$.
Suppose $X$ is ULC. Then for any metric space $Z$, $Y\times Z$ is a tight extension of $X\times Z$.
\end{lemma}

\begin{proof} Take a closed binary cover $X\times Z=F^+\cup F^-$. We should
prove that $\overline{F^+}^{Y\times Z}\cap \overline{F^-}^{Y\times Z}\subseteq\overline{F^+\cap F^-}^{Y\times Z}$. Suppose the contrary. Then there is a point
\begin{equation}\label{closure}
(y,z)\in \overline{F^+}^{Y\times Z}\cap \overline{F^-}^{Y\times Z}
\end{equation}
and a neighborhood $W$ of $(y,z)$ such that
\begin{equation}\label{empty}
W\cap (F^+\cap F^-)=\emptyset.
\end{equation}
Without loss of generality, we may assume that $W=U\times V$. Let $(x_n^+,z_n^+)$ be a sequence in $F^+$ converging to $(y,z)$ and
let $(x_n^-,z_n^-)$ be a sequence in $F^-$ converging to $(y,z)$. Choose $\eps>0$ such that for all sufficiently large $n$ we have
$B^X_\eps(x_n^+) \subset U$ and $B^X_\eps(x_n^-) \subset U$. By taking a subsequence we may assume that these inclusions hold for every $n$.

Since $X$ is ULC, by Lemma \ref{unif.l.c.} there is $\delta>0$ and connected sets $W_{x_n^+}$ and $W_{x_n^-}$ with $B^X_\delta(x_n^+) \subset W_{x_n^+} \subset B^X_\eps(x_n^+) \subset U$ and $B^X_\delta(x_n^-) \subset W_{x_n^-} \subset B^X_\eps(x_n^-) \subset U$. For $n$ large enough $z_n^+$ and $z_n^-$ lie in $V$. The connected sets $W_{x_n^+} \times \{z_n^+\}$ and $W_{x_n^-} \times \{z_n^-\}$ are disjoint from $F^+ \cap F^-$ 
and therefore 
$W_{x_n^+} \times \{z_n^+\} \subset F^+$ and $W_{x_n^-} \times \{z_n^-\} \subset F^-$ for every sufficiently large $n$. On the
other hand for all sufficiently large $n$ we have $W_{x_n^+} \cap W_{x_n^-} \supset B^X_\delta(x_n^+) \cap B^X_\delta(x_n^-) \supset
B^X_{\delta/2}(y) \neq \emptyset$. So there is $x\in X$ such that, for all large $n$, $x \in W_{x_n^+} \cap W_{x_n^-}$. Now $(x,z_n^+)
\in F^+$ tends to $(x,z)$ and $(x,z_n^-) \in F^-$ tends to $(x,z)$. So $(x,z) \in F^+ \cap F^-$. This contradicts the fact that $(x,z) \in W$.
\end{proof}

\begin{remark}\label{UC-WULC} 
Let $X$ be a UC space and let $Y$ be a compact ULC space. Then $X\times Y$ is a WULC. Indeed, let $(x_n,y_n) \in X\times Y$ and $(x'_n, y'_n)\in X\times Y$ be
two discrete adjacent sequences. Since $Y$ is compact it follows that  $(x_n)_{n\in \N}$ and $(x'_n)_{n\in \N}$ are discrete adjacent  sequence in $X$. Since $X$ is $UC$, by Example \ref{Ex_UC} we have $x_n = x'_n$ for all but finitely many $n$. Now using the assumption that $Y$ is ULC we get a   sequence $(C_n)_{n\in\N}$ of connected subsets of $Y$ and $k\in \N$ such that $\lim _{n\to \infty} \mbox{diam} C_n=0$, and $a_{n+k}\in C_n, b_{n+k}\in C_n$ for every $n\in\N$. Since
    $x_n = x'_n$ for all big enough $n$, the connected sets    $\{x_n\}\times C_n$ witness WULC for  $X\times Y$.
\end{remark}

A stronger result will be given below (see the WULC option of Theorem \ref{ALC/WULC_prod}). 

\section{Necessary conditions for straightness of  finite products}

\subsection{The ULC/precompact dichotomy of binary products}

\begin{theorem}\label{Thm_loc_connN} If $X\times Y$ is ALC, then $X$ is ULC or $Y$ is compact.
\end{theorem}

\begin{proof} Assume that $Y$ is not compact. We shall prove that
$X$ is locally connected. Then, being straight, it is also ULC.

    Let $z \in X$ and let $U$ be a \nbd of $z$.  We need to find a connected \nbd $Q$ of $z$ contained in $U$.  Let $Q=Q_U(z)$ be the  quasi-component of $z$ in $U$, namely the intersection of all the (relatively) clopen subsets of $U$ containing $z$. It suffices to prove that $Q$ is open. Indeed, in such a case $Q$ will be a minimal
    clopen subset of $U$, and therefore it will be connected.  Assume for a contradiction that $Q$ is not a \nbd of some $x\in Q$.  Then there exists a sequence $x_n\to x$ in $U$ such that $x_n\not \in Q$ for every $n\in\omega$.  Since $Q=Q_U(z)$ is the quasi-component of
    $x$ as well, this implies that $x_n$ and $x$ cannot be quasi-connected by a set contained in $U$. It follows that there is
    $\delta > 0$ such that $\hat d (x_n, x)>\delta$ for every $n$, where
    $d = d_X$ is the metric on $X$ (it suffices to take $\delta$ smaller than the distance between $x$ and the complement of $U$).  Now since
    $Y$ is not compact it contains a discrete sequence $(r_n)_{n\in \N}$.  Let $u_n = ( x_n, r_n)$ and $v_n = (x, r_n)$.
    From $\hat d (x_n, x)>\delta$ we deduce:

\medskip

\underline{Claim:} $\hat d_{X\times Y}(u_n, v_n) > \delta$.

    In fact suppose for a contradiction that $I\subset X \times Y$ is a set of diameter $\leq \delta$ which quasi-connects $u_n$ and $v_n$.
    Its projection $\pi(I)$ on $X$ has diameter $\leq \delta$, so cannot quasi-connect $x_n$ and $x$ so $\pi(I)$ can be partitioned into clopen sets $A,B$ containing $x_n$ and $x$
    respectively.  But then $\pi^{-1}(A) \cap I$ and $\pi^{-1}(B) \cap I$ form a clopen partition of $I$ separating $u_n$ and $v_n$. This contradiction proves the claim.

\medskip

Since $u_n,v_n$ are discrete adjacent sequences, we conclude with the Claim that $X \times Y$ is not ALC, contradicting the assumptions.
\end{proof} 

 The next theorem gives an easy criterion for a finite product of metric spaces to be ALC (resp., WULC). 

\begin{theorem}\label{ALC/WULC_prod} Let $X_1, \ldots, X_n$ be metric spaces. Then $X = \prod_{i=1}^n X_i$
is ALC (resp., WULC) if and only if one of the following conditions holds:
\begin{itemize}
\item[(a)] all spaces $X_i$ are compact;
\item[(b)] all spaces $X_i$ are ULC; 
\item[(c)] one of the spaces is ALC (resp., WULC) and all other spaces are both compact and ULC.
\end{itemize}
    \end{theorem}

\begin{proof} Assume that $X$ is ALC (resp., WULC). Then clearly every space $X_i$ is ALC (resp. WULC). 
Assume some of the spaces, say $X_1$, is neither compact nor ULC. Then Theorem \ref{Thm_loc_connN} yields that $\prod_{i=2}^n X_i$ is both compact and ULC. 
This proves (c). Hence we can assume that each one of the spaces is either compact or UCL. 
If one of the spaces, say $X_1$, is non-compact, then it is ULC and by Theorem \ref{Thm_loc_connN} all spaces $X_i$, $i>1$, are ULC. Thus (b) holds true.
If there exists a space that is compact, but non-ULC, then a similar argument leads to (a).

Since both compact and ULC imply WULC (hence ALC as well), to prove the sufficiency it is enough to consider only the case (c). 
Assume that all spaces $X_i$, $i>1$, are both compact and ULC. Let $Y=\prod_{i=2}^n X_i$. Then $Y$ is a compact ULC space by Lemma \ref{prod_ULC}. 
We shall prove that $X$ is ALC (resp. WULC) when $X_1$ has the same property. 

Let $(x_n,y_n)$ and $(x_n',y_n')$ be discrete adjacent sequences in $X=X_1\times Y$. We can assume without loss of generality that $y_n\to y$ and $y_n'\to y$ for some $y\in Y$ (as $(y_n)$ and $(y_n')$ are adjacent sequence in the compact space $Y$). Now discreteness of $(x_n,y_n)$ and $(x_n',y_n')$ yields that $(x_n)$ and $(x_n')$ are discrete adjacent sequences in $X$. We can find a sequence $(I_n)$ of subsets of $X$ such that 
\begin{itemize}
  \item[(a)]  $\mbox{diam } I_n \to 0$, and 
  \item[(b$_1$)] $I_n$ quasi connects $x_n$ and $x_n'$, in case $X$ is ALC, 
  \item[(b$_2$)] $I_n$ is connected, in case $X$ is WULC. 
\end{itemize}
Since $Y$ is ULC and  $y_n\to y$, $y_n'\to y$, there is  a sequence $(C_n)$ of connected subsets of $Y$ such that $y_n, y_n'\in C_n$ and $\mbox{diam } C_n \to 0$. Let $J_n=I_n\times C_n$. Then $\mbox{diam }J_n \to 0$ and $J_n$ quasi connects $(x_n,y_n)$ and $(x_n',y_n')$ in case  
 $X$ is ALC, while $J_n$ is connected in case $X$ is WULC. Hence, $J_n$ witness ALC-ness (resp., WULC-ness) of $X$. \NB
\end{proof}

The next corollary establishes one direction (the necessary condition) of Theorem A. 

\begin{corollary}\label{Thm_loc_conn} If $X \times Y$ is straight, then $X$ is ULC or $Y$ is precompact. \end{corollary}

\begin{proof} Assume that $Y$ is not precompact. We shall prove that $X$ is ULC. The completion $\widetilde{X\times Y}=\widetilde{X}\times\widetilde {Y}$ is straight by Corollary \ref{NEWcorol}. Hence by Corollary \ref{coro_comp_try} it is ALC. Since $\widetilde {Y}$ is not compact by our hypothesis, Theorem \ref{Thm_loc_connN} yields that $\widetilde{X}$ is ULC. Now Lemma \ref{dense-super} implies that $X$ is ULC.
\end{proof}

\begin{corollary}\label{powers} If $X\times X$ is straight for a metric space $X$, then $X$ is either ULC or straight and precompact.
\end{corollary}

 It was proved in \cite[Cor. 5.11]{BDP+} that if $X \times (\omega + 1)$ is ALC, then $X$ is complete. Using Theorem \ref{Thm_loc_connN} we obtain the following stronger result:

\begin{corollary} \label{alc-compact} $X \times (\omega + 1)$ is ALC if and only if $X$ is compact.
\end{corollary}

\begin{example} \label{exYamada} As an application of the above corollary let us verify the following fact proved in \cite{Y}:
if $X \times (\omega + 1)$ is straight for some metric space $X$, then $X$ is precompact. Indeed, the completion $\widetilde{X}\times (\omega + 1)$ of $X \times (\omega + 1)$ 
is straight by Corollary \ref{NEWcorol}, so also ALC by Theorem \ref{coro_comp_try}. Now the above corollary implies that  $\widetilde{X}$ is compact, i.e., 
$X$ is precompact. Note that this fact could be obtained also directly from Corollary \ref{Thm_loc_conn}  as $(\omega + 1)$ is not ULC (see also Corollary \ref{Coro2}.)
\end{example}

The next corollary (which should be compared to Corollary~\ref{powers} and  Lemma~\ref{prod_ULC})
shows that the ALC and WULS properties are preserved by non-trivial finite powers only when the starting space is compact 
or ULC, i.e. only in cases which are trivially true.


\begin{corollary}\label{cor-prec-power} For every metric space $X$ TFAE: 
\begin{itemize}
  \item[(a)] $X\times X$ is ALC.
  \item[(b)] $X$ is either compact or ULC.
  \item[(c)] $X^n$ is WULC for every $n\in \N$.
\end{itemize}
\end{corollary}

\begin{proof}
(a) $\to$ (b) follows directly from Theorem \ref{Thm_loc_connN}, (b) 
$\to $ (c) follows from Lemma \ref{prod_ULC} and (c) $\to$ (a) follows form \AB Theorem \ref{alc-stra}. 
\end{proof}

In particular, if \AB $X \times X$ is straight for a complete metric space $X$, then $X$ is either ULC or compact.
Consequently, all finite powers of $X$ are straight.



\begin{example}\label{last_example} Let $X$ be a UC space. Then $X$ is complete, hence by Corollary \ref{coro_comp_try}  $X\times X$ is straight if and only if it is
ALC. Hence the above proposition implies that  $X\times X$ is straight if and only if $X$ is either compact or ULC. Since examples
of UC spaces that are neither either compact nor ULC exist in profusion (just take any non-compact totally disconnected UC space),
we see that  $X\times X$ need not be straight for a UC space $X$. \end{example}

\subsection{Characterization of ULC spaces via straightness of products}

In the next corollary of Corollary \ref{Thm_loc_conn} we characterize the ULC spaces in terms of straightness of various products.
Note that by Theorem  \ref{Th_loc_con_vs_ULC} a space is ULC iff it is straight and locally connected. 
 
\begin{corollary}\label{Cor_loc_conn} For a metric space $X$ TFAE:
\begin{itemize}
     \item[(a)] $X$ is ULC;
     \item[(b)] $\R\times X$ is straight;
     \item[(c)] $\N \times X$ is straight;
     \item[(d)] $Y\times X$ is straight for some non-precompact space $Y$.
\end{itemize}
\end{corollary}

\begin{proof} (a) implies both (b) and (c) by Lemma \ref{prod_ULC}. $(d)$ implies $(a)$ by Corollary \ref{Thm_loc_conn}.
 The implications (b) $\to $ (d) and (c) $\to $ (d) are obvious.
\end{proof}

\begin{corollary}\label{lift_ULC} If $f:X\to Y$ is a u.c. map allowing the lifting of adjacent  sequences, then $Y$ is ULC whenever $X$ is ULC.
\end{corollary}
\begin{proof} It suffices to note that also the map $f\times id_\N:X\times \N\to Y\times \N$ allows the lifting of adjacent sequences and then apply Corollary \AB \ref{Cor_loc_conn} and Lemma \ref{lemma_imag}.
\end{proof}

Let $C$ be the Cantor space.

\begin{corollary}\label{Coro2}\label{coroll_lc}  Let $X$ be a metric space. If the product  $X\times Y$ is straight for some non-discrete
zero-dimensional space $Y$, then $X$ is precompact and $Y$ is UC. In particular, if either $X\times C$ or $X\times (\omega +1)$ is straight, then $X$ is precompact.
\end{corollary}

\begin{proof} Precompactness of $X$ follows from Corollary \ref{Thm_loc_conn} since $Y$ is not locally connected. Since straight zero-dimensional spaces are UC, the second part follows from Corollary \ref{Coro1}.
\end{proof}

\begin{remark} Note that if the space $Y$ is discrete, then $Y$ must be uniformly discrete (by straightness). 
 
 Moreover,  zero-dimensionality of $Y$ is important in Corollary~\ref{coroll_lc}. Take the open unit disk $D$ in the
plane. Then $D$ is precompact, ULC and it is not UC. However, by Theorem~\ref{pre-ULC}, $X\times Y$ is straight for each straight space $X$.
\end{remark}

%

\section{Sufficient conditions}\label{NewSec}

\subsection{A general property of tight extensions and its
consequences}\label{suff(a)}

\AB The following theorem shows that tight extensions are preserved under finite products. Later it will be extended to 
infinite products under the additional assumption that the spaces are ULC (see Theorem \ref{pulc}). We do not know 
whether the additional assumption in the infinite case can be removed. 

\begin{theorem}\label{NewTh} Let $X,Y$ be dense tight subspaces of the metric spaces $X', Y'$. Then $X'\times Y'$ is a tight extension of $X\times Y$.
\end{theorem}

\begin{proof}
We can assume $Y= Y'$ (since a composition of two tight extensions is tight).

Let $X\times Y = A\cup B$ with $A,B$ closed in $X\times Y$. Let $\ov A$, $\ov B$ be the closures of $A,B$ in $X'\times Y$ and assume
$(\ov x, \ov y) \in \ov A \cap \ov B$. We must prove that $(\ov x, \ov y) \in \overline{A \cap B}$. Suppose this is not the case and
let $U\times V$ be an open neighborhood of $(\ov x, \ov y)$ in
$X'\times Y$ with $$A\cap B\cap (U\times V)=\emptyset.\eqno(5)$$

Now fix $(x_n,y_n) \in A$ converging to $(\ov x, \ov y)$ for $n\to \infty$, and $(x'_n,y'_n)\in B$ also converging to $(\ov x, \ov y)$.
We can assume that these two sequences lie in $U \times V$. So, by $(5)$, $(x_n,y_n) \nin B$ and $(x'_n,y'_n)\nin A$ for every $n$.

\begin{Claim} We may choose the two sequences so that $y_n' = y_n$ for every $n$. \end{Claim} 

\begin{pf} Fix $(x_n,y_n)$ and $(x_n',y_n')$ as above. If either 
$$
\liminf_n d((x_n,y_n), B \cap (X \times \{y_n\})) = 0\;\;\mbox{ or }\;\;\liminf_n d((x'_n,y'_n), A \cap (X \times \{y'_n\})) = 0
$$ then it is easy to construct two sequences
as desired. So assume that the two limits are not zero. Then there
is a positive $\eps$ such that for every $n$ 
$$
d((x_n,y_n), B \cap (X \times \{y_n\})) > \eps\;\;\mbox{ and }\;\;d((x'_n,y'_n), A \cap (X \times \{y'_n\})) >\eps.
$$ 
By choosing a subsequence we can assume that for all $n$, 
$$
d((x_n,y_n),(\ov x, \ov y))<\eps/4\;\;\mbox{ and }\;\;d((x'_n,y'_n),(\ov x, \ov y))<\eps/4.
$$ 
Choose $x\in X$ at distance $<\eps/4$ from the common limit $\ov x = \lim_n x_n = \lim_n x'_n$.
Given $n$ it then follows that $(x,y_n)$ is at positive distance (at
least $\varepsilon /2$) from $B \cap (X \times \{y_n\})$ and therefore belongs to $A$. Similarly $(x,y'_n) \in B$. The two
sequences $(x,y_n)$ and $(x,y'_n)$ have a common limit $(x, \ov y)
\in X\times Y$, and $A,B$ are closed in $X\times Y$. So $(x,\ov y) \in A \cap B$. This contradicts (5), since $(x,\ov y) \in U\times
V$.
\end{pf}

Thanks to the claim we can assume $y_n=y'_n$. Let 
$$
A_n = \{x \in U \mid (x,y_n) \in A\}\;\;\mbox{ and }\;\;B_n = \{x \in U \mid (x,y_n) \in B\}.
$$
Then by (5) $\{A_n,B_n\}$ is a clopen partition of $U\cap X$ with 
$$
x_n\in A_n\;\;\mbox{ and }\;\;x'_n\in B_n.\eqno(6)
$$
\begin{Claim} \label{infinite} For any given $x\in U\cap X$, the sets $\{n \mid x \in A_n\}$ and $\{n \mid x \in B_n\}$ cannot both be infinite.
\end{Claim}

\begin{pf} Assume that $(x,y_{n_k})\in A$ and $(x,y_{n_m})\in B$ for infinitely many $k$ and $m$. Then $(x,\ov y)=\lim_k (x,y_{n_k})=\lim_m (x,y_{n_m}) \in A\cap B$. This contradicts ($*$), since $(x,\ov y)\in U\times V$.
\end{pf}

Making use of the sequence $\{A_n,B_n\}$ of binary clopen partitions of $U\cap X$ we produce now a clopen partition of $U\cap X$ consisiting of appropriate intersections of 
the clopen sets $A_n,B_n$. To describe more conveniently these intersections we use infinite binary sequences $f\in 2^\omega$. For a given $f$ let $f|n \in 2^n$ be its initial sequence of length $n$. 
Let $2^{<\omega}=\bigcup_n 2^n$ be the set of all finite binary sequences. If $\sigma\in 2^n$ then $n = lh(\sigma)$ is the length of $\sigma$. For $\sigma \in 2^{<\omega}$ define $C_\sigma \subset U\cap X$ inductively as follows.
\begin{itemize}
  \item $C_\emptyset = U\cap X$;
  \item If $lh(\sigma)=n$, $C_{\sigma 0} = C_\sigma \cap A_n$ and $C_{\sigma 1} = C_\sigma \cap B_n$.
\end{itemize}
Note that $\{C_\sigma \mid lh(\sigma) = n\}$ is a partition of $U\cap X$ into at most $2^n$ relatively clopen sets (some $C_\sigma$ may be empty).

Now for $f\in 2^\omega$ define $C_f = \bigcap_n C_{f|n}$. Clearly $C_f$ is closed in $U\cap X$ and $U\cap X$ is partitioned by the various $C_f$. Note that for each $x\in C_f$ we have $x\in A_n$ iff $f(n) = 0$. Hence we get immediately by Claim \ref{infinite}:
$$
C_f\ne \emptyset  \;\;  \; \Longrightarrow \; \; \; f   \;\mbox{ is eventually constant}.\eqno(7)
$$   
\begin{Claim} For all $f\in 2^\omega$, $C_f$ is open in $U\cap X$. \end{Claim}

\begin{pf} If $C_f$ is not open, then there is a point $z\in C_f$ in the closure of $(X\cap U)\setminus C_f$. Choose $z_k \in (X\cap U)\setminus C_f$ converging to $z$ for $k\to \infty$. 
By (7) we can assume without loss of generality that $f$ is eventually equal to the constant $0$, i.e. there is $N$ such that $\forall n \geq N$, $f(n) = 0$. It then easily follows that $C_f \times \{\ov y\} \subset A$. Let $\sigma = f|N$, so $f = \sigma 000000\ldots$. Now take $m\geq N$. Since $C_{f|m}$ is an open neighborhood of $z$, for all $k$ sufficiently big we have
$z_k \in C_{f|m}$. So for every big $k$, since $z_k \nin C_f$, there is some $n_k \geq m$ such that $z_k \in B_{n_k}$ (let $g\in 2^\omega$ be such
that $z_k \in C_g$ and choose $n_k$ so that $g(n_k) \neq f(n_k)$). So $(z_k, y_{n_k}) \in B$. We can arrange so that $n_k$ tends to $\infty$ (since $m$ above was arbitrary). So $(z_k, y_{n_k}) \to (z, \ov y)$. But then since $B$ is closed in $X\times Y$, $(z,\ov y) \in B$, contradicting $C_f \times \{\ov y\} \subset A$.
\end{pf}

So we have proved that $U\cap X$ is partitioned into the clopen sets $C_f$. Now consider the sequence $x_n \to \ov x \in X'$.

\medskip 

\noindent {\bf Case 1.} $\{x_n \mid n\}$ intersects infinitely many $C_f$. Then by choosing a subsequence we can assume that for $n\neq m$, $x_n$ and
$x_m$ belong to different clopen sets $C_f$ and $C_g$. Let $P = \bigcup \{ C_f \mid \exists n \,:\, x_{2n}\in C_f\}$ and let $Q = (X\cap U) \setminus P$. Then 
\NB $P\cup Q$ is a clopen partition of $X\cap U$, hence $P' = P \cup (X \setminus U)$ and $Q' = Q \cup (X \setminus U)$ is a binary closed cover of $X$
with $P'\cap Q'=X \setminus U$.  Moreover,  
$$
\ov x = \lim_n x_{2n} = \lim_n x_{2n+1} \in \ov{P}\cap \ov{Q} \subset \ov{P'}\cap \ov{Q'},
$$ 
where the closures are taken in $X'$. Since $X'$ is a
tight extension of $X$, $\ov x \in \ov{P'\cap Q'}$. This is absurd since $U$ is a neighborhood of $\ov x$ disjoint from $\ov{P'\cap Q'}$.

\medskip 

\noindent {\bf Case 2.}  $\{x'_n \mid n\}$ intersects infinitely many $C_f$. Similar to Case 1.

\medskip 

\noindent {\bf Case 3.} If Case 1 does not hold $\{x_n \mid n\}$ intersects finitely many $C_f$. So there is a single $C_f$ containing infinitely many
$x_n$. Let $I \subset \N$ be the infinite set $I = \{n \mid x_n \in C_f\}$. Assuming that Case 2 does not hold, there is some $C_g$
containing $x'_n$ for $n$ ranging in an infinite subset $J$ of $I$. Moreover $g$ must be different from $f$ by (6). Let 
$$
P = C_f \cup (X \setminus U)\;\;\mbox{ and }\;\;Q = \bigcup_{h \neq f} C_h \cup (X \setminus U).
$$ 
Then $P,Q$ form a closed binary cover of $X$ and $\ov x = \lim_n x_{n} = \lim_n x_{n'} \in \ov{P}\cap \ov{Q}$, where the closures are taken in $X'$. Since $X'$
is a tight extension of $X$, $\ov x \in \ov{P\cap Q}$. This is absurd since $U$ is a neighborhood of $\ov x$ disjoint from
$\ov{P\cap Q}$.
\end{proof}

A direct application of this theorem and Theorem \ref{str-dense} implies  that finite products of precompact straight spaces are straight:

\begin{theorem}\label{yama-ales} Let $X,\ Y$ be precompact straight spaces. Then $X\times Y$ is precompact straight, too.
\end{theorem}

This establishes the sufficiency of (a) in Theorem A. In particular, Theorem~\ref{yama-ales} gives

\begin{corollary} All finite powers of a  straight space $X$ are straight whenever $X$ is precompact or ULC.
\end{corollary}

This should be compared with the limits for multiplicativity of the ALC property, given in Remark~\ref{cor-prec-power}: $X\times X$ is
very rarely ALC as $X$ must be compact or ULC to have this property.

As another immediate corollary we obtain a proof of the following criterion due to Nishijima and Yamada:

\begin{corollary}\label{yam} {\rm \cite{Y}} Let $X$ be a straight space. Then $X\times K$ is straight for each compact space $K$ if and only if $X\times
(\omega +1)$ is straight.
\end{corollary}
\begin{proof}[{I proof.}] Assume $X\times (\omega +1)$ is straight. Then $X$ must be precompact by Corollary
\ref{Thm_loc_conn}. Now Theorem~\ref{yama-ales} applies.
\end{proof}

We give also a second proof that does not make recourse to Theorem~\ref{yama-ales}:

 \begin{proof}[{II proof.}] \AB Suppose $X$ is straight, $K$ is compact and $X\times K$ is not straight. Take a binary closed cover $C^+,\ C^-$ of $X\times K$
witnessing it, i.e. there are $\eps >0$ and adjacent sequences $(u_i)_{i\in\N}$ and $(v_i)_{i\in\N}$ such that $\{u_i\}_{i\in\N}\subseteq C^+_\eps$ and $\{v_i\}_{i\in\N}\subseteq
C^-_\eps$. As $K$ is compact, we may suppose (choosing a subsequence, if necessary) that sequences $(\pi _K u_i)$ and $(\pi
_K v_i)$ converge in $K$; the limits of these sequences have to coincide. Denote this limit as $k$. Define a subspace $Z= X\times
\bigl( \{k\}\cup \{\pi _K u_i\}\cup \{\pi _K v_i\}\bigr)$. $Z$ is homeomorphic to $X\times (\omega +1)$ and witnesses non-straightness
of $X\times K$.
\end{proof}

\subsection{Characterization of precompact ULC spaces}

In the sequel we prove the sufficiency of item (c) of our Main Theorem. In doing this we obtain also a characterization of the
precompact ULC spaces in terms of straightness of products.

\begin{lemma} \label{claim1}
Let $X$ be a compact ULC metric space. Then $X\times Y$ is straight for every
straight space $Y$.
\end{lemma}
\begin{proof}
Suppose for a contradiction that $X\times Y$ is not straight for some straight
space $Y$.

Since $X\times Y$ is not straight, by Theorem \ref{uplaced} there are closed
sets $C^+,\ C^-\subseteq X\times Y$ such that $C^+\cup C^-=X\times Y$ and
$C^+,\ C^-$ are not u-placed. So there is $\eta> 0$ and a pair of adjacent
sequences $(x^n_1 , y^n_1)\in C^+$ and $(x^n_2 , y^n_2)\in C^-$ such that
\begin{equation}\label{distance}
\mathrm{dist}\bigl((x^n_i , y^n_i), C^+\cap C^- \bigr) \geq\eta_0 , \ i=1,2.
\end{equation}
Since $X$ is ULC there is $\lambda > 0$ such that for all $x,x'\in X$ with $\rho(x,x') < \lambda$ there exists a connected subset
$C_{x,x'}$ of $X$ such that $\{v,w\}\subseteq C_{x,x'}$ and $\diam (C_{x,x'}) \leq \frac \eta 4$.

We claim that $\forall y\in B^\sigma_\lambda (y^n_1)$, $B^\rho_\lambda (x^n_1)\times\{y\}$ cannot intersect both $C^+$ and $C^-$.

In fact, if for a contradiction there were $(w,y)\in C^+$ and $(v,y)\in C^-$ with $\{v,w\}\subseteq B^\rho_\lambda (x^n_1)$, then $C_{w,v}\times\{ y \}\cap (C^+\cap C^-)\not =\emptyset$,
contradicting (\ref{distance}) and proving the claim.

We can conclude that for each $n$ sufficiently large there are $y_{C^+},\ y_{C^-} \in B^\sigma_\lambda (y^n_1)$ such that $B^\rho_\lambda (x^n_1)\times
y_{C^+} \subseteq C^+$ and $ B^\rho_\lambda (x^n_1)\times y_{C^-}\subseteq C^-$ (take $y_{C^+}=y_1^n, y_{C^-} = y_2^n$).

As $X$ is compact, the sequence $(x_1^n)$ has a subsequence converging to some $x\in X$ (the corresponding subsequence of $(x_2^n)$ also converges to $x$).
Then $(\{x\}\times Y)\cap C^+$ and $(\{x\}\times Y)\cap C^-$ are closed sets which are not u-placed, contradicting the straightness of $Y$.
\end{proof}

According to Corollary \ref{Thm_loc_conn} and Lemma \ref{prod_ULC}, if $X$ is a non-precompact ULC space, then $X\times Y$ is straight iff $Y$ is ULC. The next
theorem shows that adding precompactness changes completely the situation:

\begin{theorem}\label{pre-ULC}
If $X$ is precompact and ULC, then $X \times Y$ is straight for every straight space $Y$.
\end{theorem}
\begin{proof} By Corollary \ref{completionULC} the (compact) completion $\tilde X$ of $X$ must be ULC. By Lemma \ref{claim1}, $\tilde X \times Z$ is straight for every straight space
$Z$. By Lemma \ref{claim2} $\tilde X \times Z$ is a tight extension of $X \times Z$. So by Theorem \ref{str-dense} $X \times Z$ is straight.
\end{proof}

Theorem~\ref{yama-ales} complements Theorem~\ref{pre-ULC} as it relaxes the hypothesis on the first space: instead of precompact ULC,
only precompact straight is used, however the second factor in Theorem~\ref{yama-ales} has to be not only straight, but also precompact (compact, respectively).

 In the next corollary we characterize the precompact ULC spaces as those spaces $X$ such that $X\times Y$ is straight for every straight space $Y$.

\begin{corollary}\label{COR.pre-ULC} For every metric space $X$ the following are equivalent:
\begin{itemize}
\item[(a)]  $X\times Y$ is straight for every straight space $Y$;
\item[(b)]  $X\times Y$ is straight for every complete straight space $Y$;
\item[(c)]  $X\times C$ and $X\times \N$ are straight ($C$ is Cantor space);
\item[(d)] $X$ is precompact and ULC.
\end{itemize}
\end{corollary}

\begin{proof} The implications (a) $\to $ (b) $\to $ (c) are obvious. The ULC-part of the implication (c) $\to $ (d) follows from Corollary
\ref{Thm_loc_conn}.  The implication (d) $\to $ (a) is covered by the above theorem.
\end{proof}

\begin{remark} One can replace item (c) in the above corollary by the single condition $X\times C\times \N$ is straight. Note that one cannot just remove the
condition of straightness on $X\times \N$ by leaving in (c) only ``$X\times C$ is straight'' (indeed $C \times C$ is straight, but $C$ is not ULC).  
\end{remark}

We conclude by a characterization of the larger class of precompact straight spaces by means of straightness of  finite products. Using
Corollary~\ref{Coro1} and Corollary~\ref{Coro2}, we obtain immediately:

\begin{corollary}\label{coro-yam} For a metric space $X$, the following two properties are equivalent:

\begin{description}
      \item[(i)] $X$ is straight and precompact,
      \item[(ii)] $X \times K$ is straight for every compact space $K$.
\end{description}
\end{corollary}

The implication (i) $\to$ (ii) follows from the above theorem. To prove the implication (ii) $\to$ (i) note that the straightness of
the product $X \times (\omega + 1)$ alone implies precompactness of $X$ by Corollary \ref{Thm_loc_conn}.

\begin{remark}  Corollary~\ref{Cor_loc_conn} (or Corollary~\ref{Thm_loc_conn}) explains why the restriction to precompact spaces is necessary.
Recall that if $X\times Y$ is straight for a non-precompact space $Y$ then 
$X$ is ULC, so we would face again the assumptions of
Theorem~\ref{pre-ULC}.\\
 Let us recall the following fact from \cite{BDP3}: each straight totally disconnected space is UC. In particular, all
precompact straight totally disconnected spaces are compact. Having in mind also Corollary~\ref{COR.pre-ULC}, we see that
Theorem~\ref{yama-ales} says something interesting for spaces which are neither totally disconnected nor locally connected.
\end{remark}

\section{When infinite products are straight}\label{product}

\subsection{When infinite products are ALC, WULC or ULC}\label{product1}

We start by describing the stronger ALC, WULC and ULC properties for infinite products.   
The spaces $X$ such that $X^\omega$ is ULC are described below (see Corollary \ref{auxil-con}).

We have seen that a product $X\times Y$ is ULC iff both $X$ and $Y$ are ULC. The next example shows that this fails for infinite
products.

\begin{example}\label{Xomega} There is a ULC space $X$ without isolated points such that $X^\omega$ is not straight (hence not ULC).
\end{example}
\begin{proof} The starting example is $\N$ with the uniformly discrete uniformity. Certainly, $\N$ is ULC. Consider the 
infinite product
$\N^{\omega}$. 

Put $X=\bigoplus_\N [0,1]$, i.e. $X$ is a countable discrete sum of unit intervals. Then $X$ is ULC and it has no isolated points.
Consider the map $q:X\to \N$ collapsing the ${n}^{\rm{th}}$ copy of $[0,1]$ to $n$ for every $n\in\N$. Define a map $r:X^\omega \to
\N^{\omega}$ as $r=q^\omega \colon (x_i)_{i\in \omega} \mapsto (q(x_i))_{i \in \omega}$. The space $\N^{\omega}$ is not straight:
this is witnessed by a partition into two clopen sets $A,B$ with $d(A,B)=0$. Then $r^{-1}(A)$ and $r^{-1}(B)$ define a partition of
$X^\omega$ into closed sets, and by the definition of $r$ it is easy to see that $d(r^{-1}(A), r^{-1}(B)) = 0$. So $X^\omega$ is
not straight. \end{proof}

The example suggests the following more general criterion for straightness of infinite products of ULC spaces.

\begin{lemma}\label{Criterion} For a countable family $\{X_i:i\in I\}$ of ULC spaces the product is straight only if all but
finitely many of them have finitely many connected components.
\end{lemma}

  \begin{proof} Assume that infinitely many $X_i$ have infinitely many connected components. It is not restrictive to
  assume that every $X_i$ has infinitely many connected components.  Now we need the following

  \medskip

  \underline{Claim.} If $(X,d)$ is a ULC space, then there exists  a positive $\delta$ such that any two distinct connected components
  of $X$ are at distance $\geq \delta$.

  \smallskip

  \begin{pf}
  Assume for a contradiction that for every $\delta>0$ there exists  a pair of distinct connected components $C, C'$
  of $X$ with $d(C,C')\leq \delta$. Since for $x\in C$ and $y\in C'$ there exists no connected set containing both $x$ and $y$, we
conclude that $X$ is not ULC, a contradiction. This proves the Claim.
  \end{pf}

  \medskip

  By the Claim every $X_i$ admits a uniformly continuous  surjective map $f_i:X_i\to \N$. Let $f:X=\prod_i X_i\to \N^\N$ be
  the product map. Then $f$ is uniformly continuous and allows
  for lifting of adjacent sequences. Hence by Lemma \ref{lemma_imag}  $\N^\N$ is straight, a contradiction.
  \end{proof}

One can ask whether an infinite product of ULC
spaces is straight  precisely when the necessary condition from the above lemma is  satisfied.
It turens out that this fails even for infinite powers. A counter-example 
to this effect is given in Example \ref{EXA} below. 

 Note that the property ULC was necessary in order to establish the necessity  of the condition in Lemma \ref{Criterion}. 
 Straightness of an infinite product of {\em straight}
spaces does not lead to the same condition ($C^\omega$ is straight, 
even compact, with infinitely many connected components). Indeed, the Claim does not hold
for {\em straight} spaces (an example of an infinite straight precompact group with trivial connected components is given in
\cite{DP_STG}). (If Question \ref {Ques-inf} has a positive answer, then 
this condition is not necessary since then the infinite power of every 
precompact straight space would be straight.) 

Note that the next theorem covers item (a) of Theorem B.

\begin{theorem}\label{infpr-ULC} If $X_n$ is a metric space for every $n$, then for the space $X=\prod_n X_n$ the following are equivalent:
\begin{itemize}
     \item[(a)] $X$ is ULC;
     \item[(b)] each space $X_n$ is ULC and all, but finitely many, spaces are connected.
\end{itemize}
\end{theorem}

\begin{proof} (a) $\to $ (b) As $X$ is ULC, then each $X_n$ is necessarily ULC by Lemma \ref{prod_ULC}.

   Assume that infinitely many spaces $X_{n_k}$ are disconnected. Then there exists a clopen non-trivial partition
$X_{n_k}=A_k\cup B_k$. As $X_{n_k}$ is straight, $d(A_k,B_k)>0$. So for every $k\in \N$ the characteristic function $f_k:X_{n_k}\to
\{0,1\}$ of $A_k$ is u.c., so also $f=\prod_k f_k:X'=\prod_kX_{n_k}\to \{0,1\}^\omega$ is u.c. Obviously this map
allows lifting of adjacent sequences. On the other hand, $X'$ is a direct summand of the ULC space $X$, so $X'$ is ULC again by Lemma
\ref{prod_ULC}. This implies that $\{0,1\}^\omega$ is ULC by Corollary \ref{lift_ULC} (as an image of the ULC space $X'$), a
contradiction. Hence only finitely many $X_n$ can be disconnected.

(b) $\to $ (a) We have to show that $X$ is ULC. For every positive $\eps$ there exists $n_0$ such that all $X_n$ with $n\geq n_0$ are
connected and for $Z=\prod_{k=1}^{n_0}X_k$, $W=\prod_{k>n_0} X_k$, the factorization $X=Z\times W$ and for the projections  $p:X \to Z$
and $q:X \to W$ one has $\diam (\{z\}\times W) \leq \eps/2$ for each $z\in Z$. We have seen already that $Z$ is ULC. Hence there exists
$\delta>0$, such that for $z,z'\in Z$ with $d_Z(z,z')<\delta$ there exists a connected set $C$ in $Z$ containing both points and having
diameter $\leq \eps/2$. Let $x=(z,w), x'=(z',w')\in X=Y\times W$. If $d_X(x,x')<\delta$, then also $d_Z(z,z')<\delta$, so there exists a
connected set $C$ in $Z$ as above. Then $C'=C\times W$ is a connected set of $X$ containing both points $x,x'$ and having diameter $\leq \eps$.
\end{proof}

%
%
%
\if
    Hence for all $n$'s outside a finite set of integers, we find connected sets $C_n$ in $\prod_{k=1}^{n_0}X_k$ containing $a_n, b_n$ with $C_n\leq
\eps/2$. Then $B_n=C_n\times \prod_{k>n_0} X_k$ is a connected set
containing $x_n$ and $y_n$ such that $\diam (B_n) \leq \eps$. This
proves that $d^*(x_n,y_n)\to 0$ in $X$. \fi
%

\begin{remark}\label{infpr-ULC-rem}
Note that under the assumption of (b) $X$ is locally connected. Hence
   uniform local connectedness is equivalent to straightness for $X$.  
   
   There exists a compact ULC space $X$ (say $X=[0,1]\cup [2,3]$), such that $X^\omega$ is straight
  (actually, compact), but not ULC. 
Hence we deduce that straightness alone of $X$, provided all spaces $X_n$ are 
ULC, is not sufficient to imply $X$ is ULC in the above theorem. 
\end{remark}

The following corollary describes the metric spaces having their countably infinite power ULC.

\begin{corollary}\label{auxil-con} Let $X$ be a metric space. Then TFAE
\begin{itemize}
\item  $X^\omega$ is ULC;
\item   $X$ is connected and ULC.
\end{itemize}
\end{corollary}

For a connected and locally connected space $X$ the power $X^\omega$ is also locally connected and connected. So the straightness of $X^\omega$ from Corollary~\ref{auxil-con}
   is then equivalent to ULC.

If $X^\omega$ is straight then $X$ need not be ULC even if $X$ has no isolated points (take the Cantor set).

\begin{theorem}\label{ALC/WULC_Inf-prod} \NB Let $X_1, \ldots, X_n, \ldots $ be metric spaces. Then $X = \prod_{i=1}^\infty X_i$ is ALC (resp., WULC) if and only if 
one of the following conditions holds:
\begin{itemize}
\item[(a)] all spaces $X_i$ are compact;
\item[(b)] all spaces $X_i$ are ULC and all but finitely many of them are connected;
\item[(c)] one of the spaces is ALC (resp., WULC) and all other spaces are both compact, ULC and all but finitely many of them are connected. 
\end{itemize}
\end{theorem}

\begin{proof} The necessity follows from Theorem \ref{ALC/WULC_prod} and Theorem \ref{infpr-ULC}.

For the sufficiency consider three cases. In case (a) $X$ is compact, so WULC (and ALC). If (b) holds true, then $X$ is ULC by  Theorem \ref{infpr-ULC}.  
Finally, if (c) holds true, then cases (a) and (b) apply along with Theorem \ref{ALC/WULC_prod}. 
\end{proof}

Since UC spaces are both straight and complete, they are ALC by Theorem \ref{coro_comp_try}. This is why one is tempted to connect
the above theorem to the following old result of Atsuji \cite{A}:  a product $X = \prod_{i=1}^\infty X_i$ of metric spaces
is $UC$ if and only if one of the following conditions holds:
\begin{itemize}
\item[(i)]  each $X_n$  is compact or
 \item[(ii)] all but finitely many $X_n$ are one-point spaces and either all are uniformly isolated or all are finite except for one which 
is a $UC$-space. 
 \end{itemize}

Of course, the  sufficiency is obvious. For the necessity it suffices to note that every non-compact UC space $X$
has an infinite closed uniformly discrete set  $D$. If a metric space $Y$ has a non-isolated point $y$, then the product $X\times Y$ contains
a closed subset, namely $D\times \{y\}$, that is uniformly discrete and contained in the subspace $(X\times Y)'$ of non-isolated points of $X\times Y$.
Consequently, $(X\times Y)'$ is not compact and hence $X\times Y$ is not UC. Hence the product $X\times Y$ can be a UC space only when $Y$ is uniformly discrete. 
Moreover, if $X$ is not discrete this occurs precisely when $Y$ is finite. 
Since an infinite product can be uniformly discrete precisely when all but finitely many of the spaces are singletons and the remaining (finitely many) spaces are 
uniformly discrete, this shows the  necessilty of (ii). 


\subsection{Straightness of infinite products}

 We next show that one direction of Theorem A (the necessary condition) remains valid also in the case of countable products: 

\begin{proposition}\label{NecessInf} Let $\{X_i:i\in I\}$ be a countable family of metric spaces. If the product $X=\prod_{i\in I} X_i$ is straight then all spaces $X_i$
are straight  and one of the following three cases occurs
   \begin{itemize}
    \item[(a)] all $X_i$ are ULC and all but finitely many spaces are connected (i.e., $X$ is ULC);
    \item[(b)] all $X_i$ are precompact;
    \item[(c)] all but one of the spaces are both ULC and precompact, and all but finitely many spaces are connected.
\end{itemize}
\end{proposition}

\begin{proof} For every $j\in I$ let $Y_j=\prod\{X_i: i\in I\setminus \{j\}\}$.  Then one can write $X=X_j\times Y_j$ (actually, these spaces are uniformly homeomorphic).  
Assume that $X_j$ is not ULC for some $j\in I$. Then the straightness of $X$ yields $Y_j$ is precompact,
by Corollary \ref{Thm_loc_conn}. So all spaces $X_i$, $i\in I\setminus  \{j\}$, are precompact. If $X_j$ is precompact too, then
we get (b). If $X_j$ is not precompact, then $Y_j$ is ULC by Corollary \ref{Thm_loc_conn}. Hence 
all spaces $X_i$ ($i\in I\setminus  \{j\}$) are ULC and all but finitely of them are connected, by Theorem \ref{infpr-ULC}.  
Therefore, (c) holds true.

Now assume that both (b) and (c) fail. Then all spaces $X_i$ are ULC by the above argument. Moreover, if $X_j$ ($j\in I$) is a space that fails to be precompact, then
 $Y_j=\prod\{X_i: i\in I\setminus \{j\}\}$ is ULC by Corollary \ref{Thm_loc_conn}. Hence again by Theorem \ref{infpr-ULC} all but finitely many spaces are connected. 
\end{proof}

\begin{remark} \label{NewRemark} It was already proved in Theorem \ref{infpr-ULC} that item (a) is equivalent to the ULC
property of the infinite product. 
Let us see that (c) is also a sufficient condition for straightness. Indeed, if for some $j\in I$ all spaces $X_i$, $i\in I\setminus \{j\}$,
are both ULC and precompact and all  but finitely many spaces are connected, then the space $Y_j=\prod\{X_i: i\in I\setminus \{j\}\}$
is precompact and ULC by Theorem \ref{infpr-ULC}. Now Theorem \ref{pre-ULC} implies that   $X=X_j\times Y_j$ is straight. 

 What remains open is establishing sufficiency of (b) (see Question \ref{Ques-Inf}). 
\end{remark}

For powers we have the following:

\begin{corollary} If a power $X^\omega$ is straight, then either $X$ is precompact or $X^\omega$ is ULC.
\end{corollary}

\begin{proof} Assume $X$ is not precompact. Then also $X^\omega$ is not precompact.   Since $X^\omega$ is uniformly homeomorphic
to $X^\omega\times X^\omega$, we conclude with Corollary \ref{powers} that $X^\omega$ is ULC.
\end{proof}

The following results were found among the hand written notes of our late co-author Jan Pelant after his death:

\begin{theorem}\label{PM1} Let $X_n $ be ULC and precompact for each $n$, then $\Pi_n X_n$ is straight.
\end{theorem}



To prove the theorem we need the following theorem of independent interest. 

\begin{theorem}\label{pulc} Let $X_i$ be a dense ULC subset of $Y_i$ for each $i\in \N$. Then $\Pi_i Y_i$ is a tight extension of $\Pi_i X_i$.
\end{theorem}

\begin{proof} For every $m$ let $\pi_m'\colon \prod_i Y_i \to \prod_{i\leq  m}Y_i$ and $\pi_m'':  \prod_i Y_i \to \prod_{i> m}Y_i$ be the projections. 

Let $A,B$ be closed subsets of $\Pi_i X_i$ with $A\cup B = \Pi_i X_i$. Let $\ov A, \ov B$ be the closures of $A,B$ in $\Pi_i Y_i$.
Let $f\in \Pi_i Y_i$ be such that $f\in \ov A \cap \ov B$. We must show that $f\in \overline{A\cap B}$. If this is not the case there
is an open neighborhood $U$ of $f$ with $U\cap A \cap B = \emptyset$. Take a smaller open neighborhood $V\subset U$ at
positive distance $\eps>0$ from the complement of $U$, namely $d(V, \Pi_i Y_i \setminus U) = \eps >0$.

By definition of the product topology we can take $V$ of the form $V=\prod_{i=0}^\infty V_i$, where each $V_i$ is open in $Y_i$ and $V_i=X_i$ for all $i>k$.

Choose $g_n \in A$ with $\lim_n g_n = f$ and $h_n \in B$ with $\lim_n h_n = f$. We can assume that $g_n\in V$ and $h_n \in V$
for every $n$. Let $f_0=\pi_0'(f)$ and  choose $\eps_0>0$ such that $B_{2\eps_0}(f_0)\subseteq V_0$. Then according to  Lemma \ref{unif.l.c.} there exists a positive $\delta_0\leq \eps_0$ such that for every $x \in X_0$ there exists a connected open set $W_x$ in $X_0$ such that $B_{\delta_0}(x)\subseteq W_x \subseteq B_{\epsilon_0}(x)$.
As $\lim_n \pi_0'(g_n)=\lim \pi_0'(h_n)=f_0$, there exists $n_0$ such that $d(\pi_0'(g_n), f_0)<\delta_0$ and  $d(\pi_0'(h_n), f_0)<\delta_0$ for all $n>n_0$.
Hence  $d(\pi_0'(g_n), \pi_0'(h_m))<2\delta_0$ for all $m,n >n_0$. Consequently, $W_{\pi_0'(g_n)}\cap W_{\pi_0'(h_m)}\ne \emptyset$ and $W_{\pi_0'(g_n)}\cup W_{\pi_0'(h_m)}
\subseteq B_{2\eps_0}(f_0)\subseteq V_0\cap X_0$ for all $m,n >n_0$. Therefore, 
$$
C_0=\bigcup_{n >n_0} W_{\pi_0'(g_n)}\cup W_{\pi_0'(h_n)}\subseteq V_0\cap X_0
$$ 
is an open connected set and  $\pi_0'(g_n), \pi_0'(h_n) \in C_0$ for all $n >n_0$. 
Suppose for some $r\in \N$ we have constructed a sequence of natural numbers $n_0\leq \ldots \leq n_r$ and a sequence of  connected open sets $C_i \subseteq X_i\cap V_i$ for each $i\leq r$ (so, $C_0\times \ldots \times C_k \subseteq V_0\times \ldots \times V_r$) such that 
$$
\pi_r'(g_n), \pi_r'(h_n) \in C_0\times \ldots \times C_r \;\;\mbox{ for all }\;\;n\geq n_r.
$$ 
For the inductive step, arguing as above, choose $\eps_{r+1}>0$ such that $B_{2\eps_{r+1}}(f_{r+1})\subseteq V_{r+1}$ and find  
 using Lemma \ref{unif.l.c.} (as $X_{r+1}$ is ULC) a connected open set $C_{r+1} \subseteq V_{r+1}\cap X_{r+1}$ and $n_{r+1}\geq n_r$ such that $\pi_{r+1}'(g_n), \pi_{r+1}'(h_n) \in C_0\times \ldots \times C_{r+1}$  for all  $n\geq n_{r+1}$. Let $C= \prod_i C_i$. Then $C\subseteq V \subseteq U$.  

\medskip 

Let $s\in C$ and without loss of generality suppose that $s\in A$. Then $s\nin B$, so $U\setminus B$ is a neighborhood of
$s$. It then follows by the definition of the product topology that
there is $m \geq k$ and non-empty open sets $W_i \subseteq C_i$ with $s \in W=W_0 \times \ldots \times W_m \times
\Pi_{i>m}X_i \subseteq U\setminus B \subseteq A$. Choose $n\geq n_m$, hence $h_n \in C_0 \times \ldots \times C_m \times
\Pi_{i>m}X_i$. Let $t\in \Pi_i X_i$ with $\pi_m'(t)=\pi_m'(s)$ and $\pi_m''(t)=\pi_m''(h_n)$. The connected set 
$$
Q = C_0\times \ldots \times C_m \times \{\pi_m''(h_n) \} \subseteq \Pi_n X_n
$$ 
meets $B$ as $h_n \in Q$. Since $Q\cap W\ne \emptyset$  and $W \subseteq A$, it follows
that $Q\cap A\ne \emptyset$ as well. But then $Q$ is the disjoint union of
the non-empty relatively closed sets $Q \cap A$ and $Q \cap B$,
contradicting the fact that $Q$ is connected.
\end{proof}

\bigskip 
\noindent{\bf Proof of \ref{PM1}.} Let $X_n $ be ULC and precompact for each $n$, then $\Pi_n X_n$ is straight by Theorem \ref{pulc} and Theorem \ref{str-dense}.  \hfill 
$\Box$
 
\bigskip 

\noindent {\bf Proof of Theorem B.} Item (a) of the theorem was proved in Theorem \ref{infpr-ULC}. 

\medskip

(b) If $X$ is ULC, then  it is also striaght. If each $X_n$ is precompact, then $X$ is straight by Theorem \ref{PM1}. This proves the implication $(b_2)\to (b_1)$. 

\medskip

To prove the implication $(b_1)\to (b_2)$ suppose that $X$ is straight, but not all $X_n$ are precompact. We must show that $X$ is ULC.  According to Proposition \ref{NecessInf} 
either $X$ is ULC, or item (c) of the proposition holds true, i.e.,  all but one of the spaces are both ULC and precompact, and all but finitely many spaces are connected.
By Theorem  \ref{infpr-ULC} the product $X$ is ULC. \hfill 
$\Box$

\bigskip


This completely settles the case of infinite products of ULC spaces. We end up with an example. 

\begin{example}\label{EXA} According to the Corollary from the Introduction, 
 $X^\omega $ is straight for a ULC space $X$ iff $X$ is either 
connected or precompact.   Let $X=R_2 \cup  R_2$, where each $R_i$ is a copy of the reals, each $R_i$ carries the 
usual metric and $d(R_1,R_2)>0$. Then $X$ is ULC and neither precompact
nor connected. Hence $X^\omega$ is not straight. \end{example}

\section{Open questions}

We have described when infinite products of ULC spaces are again ULC or straight (Theorem~\ref{infpr-ULC}). The case of 
precompact spaces is still open, so we start with the following still unsolved

\begin{question}\label{Ques-inf} Let $X$ be a precompact straight space. Is the infinite power $X^\omega$ necessarily straight?
\end{question}

More generally:

\begin{question}\label{Ques-Inf} Let $X_n$ be a precompact straight space  for every $n\in \N$. Is the infinite product $\prod_n X_n$ necessarily straight?
\end{question}

It is easy to see that a positive answer to this question is equivalent to a positive answer to item (b) of the following general question (i.e., the version of Theorem \ref{pulc} for products of {\em precompact} spaces):

\begin{question}\label{ppstr} Let the metric space $Y_i$ be a tight extension of $X_i$  for each $i\in \N$. 
\begin{itemize}
\item[(a)] Is $\Pi_i Y_i$ \AB a tight extension of $\Pi_i X_i$.
\item[(b)]  What about {\em precompact} metric spaces $Y_i$? 
\end{itemize} 
\end{question}

As far as the more general part (a) is concerned we recall the following well known facts that give another motivation for 
the question.  
The class ${\mathcal P}$ of all perfect maps in the category of topological spaces is known to be determined by the property
$$f\in {\mathcal P} \;\;\Longleftrightarrow  \;\;f\times \mbox{id}_Y\in
{\mathcal P}\;\mbox{ for every Hausdorff space }\;Y.\eqno(*)$$ Moreover, 

\begin{itemize}
   \item[(a)] ${\mathcal P}$ is closed under composition \cite[Corollary 3.7.3]{E};
   \item[(b)] ${\mathcal P}$ is closed under arbitrary products (\cite[Theorem 3.7.7]{E}, this is the celebrated Frol\' \i k's theorem);
   \item[(c)]  ${\mathcal P}$  is left and right cancelable (i.e., if $fg\in {\mathcal T}$, then $f\in {\mathcal T}$ and $g\in {\mathcal P}$ \cite[Proposition 3.7.10]{E}). 
\end{itemize}

The class ${\mathcal T}$ of tight embeddings in the category of metric spaces has similar properties. 
Indeed, obviously ${\mathcal T}$ is closed under composition and ${\mathcal T}$ is left and right
cancellable. Moreover, ${\mathcal T}$ is closed under finite products by Theorem \ref{NewTh}.
Using this one can check that $ {\mathcal T}$ has also the property ($*$) for all metric spaces $Y$ (with ${\mathcal P}$ replaced by $ {\mathcal T}$). 
What is not clear is whether the full counterpart of (b) for countably infinite products is available for $ {\mathcal T}$ (this is Question \AB \ref{ppstr} (a);
note that countably infinite products are the limit one should stay in while working with metric spaces).  According to Theorem \ref{pulc} this is true if the domains of the maps are ULC. 
This motivates our hope, that in analogy with the class of perfect maps, also ${\mathcal T}$ is closed under {\em infinite} products, i.e., Question \AB \ref{Ques-Inf} has a positive answer. 

Theorem \ref{str-dense} gives a criterion for straightness of a dense subspace $Y$ of a straight space $X$ in terms of properties of the  embedding $Y\hookrightarrow X$
(namely, when $X$ is a tight extension of $Y$).  The counterpart of this question for {\em closed} subspaces is somewhat unsatisfactory. 
We saw that uniform retracts (Corollary \ref{u.retract}), 
clopen subspaces (Corollary \ref{clopen}), 
as well as direct summands, of straight spaces are always straight (Corollary \ref{Coro1}). 
On the other hand, closed subspaces even of ULC spaces may fail to be straight (see Example \ref{EmbedULC}). 
Another instance when a closed subspace of a straight space fails to be straight is given by the following fact proved in \cite{BDP3}:
 the spaces $X$ in which {\em every} closed subspace is straight are precisely the UC spaces \cite{BDP3}.  Hence every straight space that is not UC has closed non-straight subspaces. This motivates the following general

\begin{problem}\label{QuesProd} Find a sufficient condition ensuring that a closed subspace $Y$ of a straight space $X$ is still straight. 
\end{problem}

\begin{question} Generalize the results on straight spaces from the category of metric spaces to the category of uniform spaces. \end{question}


\end{document}